\documentclass[a4paper,12pt]{article}

\usepackage{amsmath,amsthm,amssymb,amscd}
\usepackage{amssymb}
\usepackage{amsthm}
\usepackage{amsmath, amsfonts}
\usepackage{latexsym}
\usepackage{authblk}
\usepackage{epsfig}
\usepackage{amscd}
\usepackage{graphicx,wrapfig}
\usepackage{subcaption}
\usepackage[thinlines,thiklines]{easybmat}
\usepackage[table]{xcolor}
\usepackage{hyperref}
\usepackage{float}
\hypersetup{
    colorlinks=true,
    linkcolor=blue,
    filecolor=cyan,   
    urlcolor=magenta,
    citecolor=red,
    pdftitle={Overleaf Example},
    pdfpagemode=FullScreen,
    }
\def\natu           {\mathbb N}
\def\inte 		{\mathbb Z}
\def\real		{\mathbb R}

\def\comp		{\mathbb C}

\def\R		{\cal R}
\def\L		{\cal L}

\def\lra		{\longrightarrow}

\title{On embeddability of Coxeter groups\\ into the Riordan group}
 
\author[1]{Tian-Xiao He}
\author[2]{Nikolai A. Krylov}
\affil[1]{Department of Mathematics, Illinois Wesleyan University, 
1312 Park Street, Bloomington IL 61702, USA 

email: the@iwu.edu}

\affil[2]{Department of Mathematics, Siena College, 

515 Loudon Road, Loudonville NY 12211, USA 

email: nkrylov@siena.edu}

\date{}

\begin{document}

\newtheorem{theorem}{Theorem}
\newtheorem{lemma}[theorem]{Lemma}
\newtheorem{claim}[theorem]{Claim}
\newtheorem{Corollary}[theorem]{Corollary}
\newtheorem{conj}[theorem]{Conjecture}
\newtheorem{prop}[theorem]{Proposition}
\newtheorem{question}{Problem}
\theoremstyle{definition}
\newtheorem{definition}[theorem]{Definition}
\newtheorem{example}[theorem]{Example}
\numberwithin{equation}{section}

\newcommand{\F}{\mbox{$\mathcal{F}$}}

\maketitle
\begin{abstract}
We discuss examples of linear representations of finite groups as subgroups of the 
Riordan group. In particular, we show that the symmetric group of degree three   
has no faithful representation as a subgroup of the Riordan group over the complex 
numbers, but can be embedded as a subgroup of the Riordan group over a field 
of characteristic three.
\end{abstract}

\noindent {\it 2020 Mathematics Subject Classification}: 05A05, 20C30,05E16, 20F05, 20F55, 20H20

\noindent {\it Keywords}: Coxeter group, Riordan group, Riordan array, Involution, 
Representations of a finite symmetric group.


\section{Introduction}

The Riordan group was introduced by Shapiro et al. \cite{SGWW} in 1991. The elements 
of this group, called Riordan arrays, are proved to be useful in evaluating combinatorial 
sums in a uniform and constructive way. A Riordan array is defined as a pair of formal 
power series with coefficients in an arbitrary field $\mathbb K$ (e.g. $\comp$). When 
there are no zeros on the main diagonal, we say that the Riordan array is {\sl proper}.

For example, the alternating Pascal's matrix, where the $n$-th row consists of the 
signed binomial coefficients (assuming that ${n\choose k} = 0$, if $k>n$) 
\begin{equation}
\label{Pascal1}
P_1:=
\left((-1)^k {n\choose k}\right)_{n,k\geq 0}
= \begin{pmatrix}
1 & \phantom{-}0 & 0 & \phantom{-}0 & \cdots\\
1 & -1 & 0 & \phantom{-}0 & \cdots\\
1 & -2 & 1 & \phantom{-}0 & \cdots\\
1 & -3 & 3 & -1 & \cdots\\
\vdots & \vdots & \vdots & \vdots & \ddots
\end{pmatrix},
\end{equation}
is a famous element of this group. Each column of $P_1$ consists of the coefficients 
of a formal power series (f.p.s.): the zeroth column is given by the geometric series $1/(1-t)$, 
and the $k$-th column, for an arbitrary $k\geq 0$, is given by the coefficients of 
$(-t)^k/(1-t)^{k+1}$. The group operation is the usual row-by-column matrix multiplication, 
and one quickly notices that the square of $P_1$ will give us the infinite identity 
matrix with ones on the main diagonal and zeros everywhere else, i.e., $I$, the identity 
of the Riordan group. Thus, $P_1$ is an element of order 2, which is called an involution. 

More formally, let us consider the set of all formal power 
series (f.p.s.) $\F = {\mathbb K}[\![$$t$$]\!]$ in indeterminate $t$ with coefficients in $\mathbb K$; 
the \emph{order} of $f(t)  \in \F$, $f(t) =\sum_{k=0}^\infty f_kt^k$ ($f_k\in {\mathbb K}$), 
is the minimal number $r\in{\mathbb N_0=\{0\}\cup\natu}$ such that $f_r \neq 0$. 
Denote by $\F_r$ the set of formal power series of order $r$. Let $g(t) \in \F_0$ and 
$f(t) \in \F_1$; the pair $\bigl(g(t) ,\,f(t)\bigr)$ defines the {\em (proper) Riordan array} 
$$
A =(a_{n,k})_{n,k\geq 0}=\bigl(g(t) ,\,f(t)\bigr)
$$ 
having
\begin{equation}\label{1}
a_{n,k} = [t^n]g(t) f(t)^k,
\end{equation} 
where $[t^n]h(t)$ denotes the coefficient of $t^n$ in the expansion of a f.p.s. $h(t)$. 
The reader will find all the main properties of the functionals $[t^n]$ in the {\sl Monthly's} 
paper by Merlini et al. \cite{Merlini} from 2007. It is a classical fact now that the set of all 
such Riordan arrays forms a multiplicative group, called the Riordan group, and 
denoted by ${\cal R}({\mathbb K})$ 
(see, for example, \cite{Barry1}, \cite{Cameron2}, 
\cite{Davenport1}, \cite{Shapiro2}, \cite{SGWW}). 
The group operation $*$ in ${\cal R}({\mathbb K})$ (will be omitted below for brevity) 
is written in terms of the pairs of the f.p.s. as 
\begin{equation}
\label{operation1}
\bigl(g_1(t) ,\,f_1(t)\bigr) * \bigl(g_2(t) ,\,f_2(t)\bigr) = \bigl(g_1(t) g_2(f_1(t)),\,f_2(f_1(t))\bigr),
\end{equation}
with the Riordan array $I = (1,\,t)$ acting as the group identity. The inverse of the 
Riordan array $\bigl(g(t) ,\,f(t)\bigr)$ is the pair
$$
\bigl(g(t) ,\,f(t)\bigr)^{-1} = \left(\frac{1}{g(\bar{f}(t))},\,\bar{f}(t)\right),
$$
where we used the standard notation $\bar{f}(t)$ for the compositional inverse of $f(t)$. 
Thus, $\bar{f}(f(t)) = t$ and $f(\bar{f}(t)) = t$. The Riordan group ${\cal R}({\mathbb K})$, is the 
semidirect product of two proper subgroups: the Appell subgroup and the associated 
(or Lagrange) subgroup (\cite{Barry1}, \cite{Shapiro2}). The Appell subgroup $\cal A$ is 
abelian, normal, and consists of the Riordan arrays $\bigl(g(t) ,\,t\bigr)$.
The Lagrange subgroup $\cal L$ consists of the Riordan arrays $\bigl(1 ,\,f(t)\bigr)$. 

If $G$ is a multiplicative group with the identity $e$, an element $g\in G$ 
is called an {\sl involution} if it is of order $2$, i.e., if $g^2=e$. Thus, 
$(g,\,f)\in {\cal R}({\mathbb K})$ is an (Riordan) involution if 
$(g,\,f)^2=\bigl(g\cdot g(f),\, f(f)\bigr) = (1,\,t)$. Namely, $g(f) = g\circ f=1/g$ and $f=\bar f$. 
The alternating Pascal's matrix $P_1$ is an example of such involution. An element $(g,\,f)$ 
in the Riordan group is said to be a pseudo-involution if and only if 
$$
(g,\,f)(1,\,-t) =(g,\,-f) ~~~ \mbox{or} ~~~ (1,\,-t) (g,\,f)=(g(-t), \,f(-t))
$$
is an involution. The classical Pascal's matrix, $\bigl(1/(1-t), \,t/(1-t)\bigr)=
\left({n\choose k}\right)_{n,k\geq 0}$, consisting of the non-negative binomial coefficients is 
an example of a pseudo-involution. Notice that $(1,\,-t)\in \cal L$ is 
an involution which is not in the center of ${\cal R}({\mathbb K})$, while $(-1,\,t)\in \cal A$ 
is an involution, which commutes with every element of ${\cal R}({\mathbb K})$.

Since its introduction, the group ${\cal R}({\mathbb K})$ has been of great interest in 
combinatorics and by now, there are numerous results clarifying the algebraic structure 
of this group. For instance, Luz\'on, Mor\'on, and Prieto-Martinez \cite{LMP} proved that 
any element in the group ${\mathcal R}({\mathbb K})$ generated by the involutions in 
the group is the product of at most four of them. The elements of finite order, and the 
involutions and pseudo-involutions in particular, have been studied by many 
mathematicians. Here is by no means exhaustive, a list of the corresponding 
publications: \cite{Burstein}, \cite{Cameron}, \cite{CK1}, \cite{CK2}, \cite{Coh}, 
\cite{Davenport1}, \cite{HS}, \cite{LMP}, \cite{Shapiro1}.

The alternating Pascal's matrix $P_1$ in (\ref{Pascal1}) is easily generalized to 
produce infinitely many similar involutions. Take any $b\in\mathbb K$, and 
consider the generalized Pascal's matrix (see \S 2 of \cite{Cameron}) of the form 
\begin{equation}
\label{GenPascal}
P_b = \left(\frac{1}{1- bt},\,\frac{-t}{1- bt}\right).
\end{equation}
We leave it to the reader to check that $P_b$ is indeed an involution, and that 
$P_{b_1}$ and $P_{b_2}$ do not commute, as long as $b_1\neq b_2$.

It is well known that the symmetric group ${\cal S}_n$, consisting of all bijections from the 
set $\{1, 2, \ldots, n\}$ to itself, is generated by {\sl transpositions} (permutations of order 2), 
i.e. involutions. It is also known (see \cite{Coxeter}, \S 6.2) that ${\cal S}_n$ has the 
following presentation in terms of generators and relations 
\begin{equation}
\label{Sn}
{\cal S}_n\cong \langle \sigma_1,\ldots, \sigma_{n-1} ~ | ~ \sigma_i^2 = 1, \sigma_i\sigma_j = 
\sigma_j\sigma_i  ~ \mbox{for} ~ |i-j|>1,~ (\sigma_i\sigma_{i+1} )^3 = 1 \rangle, 
\end{equation}
where $\sigma_i = (i ~ i+1),~\forall i\in\{1,\ldots, n-1\}$ are the generating involutions. 
In particular, for $n=3$ we have 
\begin{equation}
\label{S3}
{\cal S}_3\cong \langle \sigma_1, \sigma_2~ | ~ \sigma_1^2 = \sigma_2^2 = 
(\sigma_1 \sigma_2 )^3 = 1 \rangle.
\end{equation}
Moreover, the elements of ${\cal S}_3$ can be naturally embedded, preserving the group operation, 
into the general linear group of $3\times 3$ matrices ${\rm GL}(3,\mathbb K)$ as
\begin{equation}
\label{S3Present}
e\leftrightarrow \begin{pmatrix}
1 & 0 & 0\\
0 & 1 & 0\\
0 & 0 & 1\\
\end{pmatrix}, ~~ \sigma_1 \leftrightarrow \begin{pmatrix}
0 & 1 & 0\\
1 & 0 & 0\\
0 & 0 & 1\\
\end{pmatrix},  ~~ \sigma_2 \leftrightarrow \begin{pmatrix}
0 & 0 & 1\\
0 & 1 & 0\\
1 & 0 & 0\\
\end{pmatrix},
\end{equation}

$$
\sigma_1\sigma_2 \leftrightarrow \begin{pmatrix}
0 & 1 & 0\\
0 & 0 & 1\\
1 & 0 & 0\\
\end{pmatrix}, ~~ \sigma_2\sigma_1 \leftrightarrow \begin{pmatrix}
0 & 0 & 1\\
1 & 0 & 0\\
0 & 1 & 0\\
\end{pmatrix},  ~~ \sigma_1\sigma_2\sigma_1 \leftrightarrow \begin{pmatrix}
1 & 0 & 0\\
0 & 0 & 1\\
0 & 1 & 0\\
\end{pmatrix}.
$$
This is an example of a faithful representation of ${\cal S}_3$ by $3\times3$ permutation 
matrices. Analogous representation of ${\cal S}_n$ exists for every $n\in\natu$ 
(see, for example, \cite{Sagan}, \S 1.3). The Riordan group over $\comp$ consists of 
infinite lower triangular matrices and has plenty of involutions, so a natural question arises: 
what are faithful representations of ${\cal S}_n$ by such Riordan arrays? Surprisingly, it turns 
out that when $n\geq 3$, there are none. It will follow from our theorem proven below, that 
for $n\geq 3$, ${\cal S}_n$ can not be embedded into the Riordan group over $\comp$.
But first note, that since $\real$ has only two roots of unity, $\pm1$,
the only real-valued Riordan arrays  of finite order, which are different from 
the identity, are the elements of order 2, and ${\cal R}(\real)$ can not contain an element $r\neq I$, 
such that $r^3 = I$. However, if we consider the Riordan arrays over $\comp$, 
then using the complex primitive roots of unity, one easily generalizes the idea of the alternating 
Pascal's matrix, and produces an element of arbitrary finite order (see \cite{CK2}). 
Here is an example suggested by an anonymous referee. Take two arbitrary multiplicative 
subgroups of $\comp\setminus\{0\}$, say $G$ and $H$. Then we can embed the direct 
product of these two groups into ${\cal R}(\comp)$ as 
$$
G\times H \cong\{(a,\,bt)\in {\cal R}(\comp) ~ | ~ a\in G, b\in H\}.
$$
The element $(a,b)\in G\times H$ is represented by the 
infinite diagonal matrix with the numbers $\{a,ab,ab^2,\ldots,ab^n,\ldots\}$ 
along the main diagonal. In particular, this way we can embed the torus 
$S^1\times S^1$ and the cylinder $S^1\times \real$, as well as numerous finite groups.

In the view of symmetry of the products, we are also inspired by the palindromes of 
(Riordan) involutions and pseudo-involutions to consider the identification of the 
product of three involutions in a palindromic form. For instance, taking two 
involutions $\sigma_1$ and $\sigma_2$, it is immediate that the palindromes 
$\sigma_1\sigma_2\sigma_1$ and $\sigma_2\sigma_1\sigma_2$ will also be involutions. 
The same statement holds true for the pseudo-involutions (see section 3 of \cite{HS}). 
For involutions $\sigma_1$ and $\sigma_2$, the identities $(\sigma_1\sigma_2)^3=1$ and
$\sigma_1\sigma_2\sigma_1 = \sigma_2\sigma_1\sigma_2$ are equivalent. Therefore, 
if we could find two non-commuting Riordan involutions $\sigma_1$ and $\sigma_2$ 
so that the palindromes $\sigma_1\sigma_2\sigma_1$ and $\sigma_2\sigma_1\sigma_2$ 
were equal, the presentation (\ref{S3}) would imply that $S_3$ is embeddable into the 
Riordan group. However, this case seems unlikely, if we try to find such involutions 
among the Pascal type Riordan arrays we have mentioned above (\ref{GenPascal}). 

In the next section, we will recall some preliminary knowledge about Coxeter groups, 
including the symmetric groups, dihedral groups, etc., and show how to embed 
two particular Coxeter groups (the Klein four-group and the infinite dihedral group) 
into ${\cal R(\comp)}$. We also prove that the embeddability of non-abelian Coxeter 
groups into the Riordan group ${\cal R(\comp)}$ does not hold in general.
In the last section, we show several examples of representations of the 
dihedral group ${\cal D}_n$ as a subgroup of ${\cal R}(\inte_n)$, and also 
classify degree-2 representations of the symmetric group ${\cal S}_3$ as 
a subgroup of the truncated Riordan group ${\cal R}_2(\inte_3)$. We will prove that 
there are exactly two such representations, one of which is equivalent to 
the degree-2 representation of ${\cal S}_3$ reduced from the standard degree-3 
representation by permutation matrices over $\inte_3$. 
A few open questions are offered at the end of this section.


\section{Embeddability into ${\mathcal R}(\comp)$}

A Coxeter group $G$ can be introduced using generators and defining relations as 
\begin{equation}
\label{Coxeter}
G\cong \langle r_1,\ldots, r_k ~|~ (r_ir_j)^{m_{ij}} = 1\rangle,
\end{equation}
where $m_{ii}=1$ and $m_{ij}=m_{ji}\geq 2$ is either an integer or $\infty$ for $i\not=j$. 
If there is no relation between $r_i$, and $r_j$, $(r_ir_j)^{m_{ij}}=1$ for any integer 
$m_{ij}\geq 2$, we assume $m_{ij} = \infty$. (see \S 5.1 of \cite{Humphreys}). 
In particular, the symmetric group ${\cal S}_n$ of order $n!$, 
and the dihedral group ${\cal D}_n$ of order $2n$ are examples of Coxeter groups. 
Recall that ${\cal S}_n$ has a presentation (\ref{Sn}), and ${\cal D}_n$ has a presentation
\begin{equation}
\label{SnDn}
{\cal D}_n \cong \langle r_1, r_2 ~|~ r_1^2 = r_2^2 = 1, ~ (r_1r_2)^n = 1\rangle.
\end{equation}
The relations $(r_ir_i)^{1} = 1$ in (\ref{Coxeter}), mean that all the generators of $G$ are involutions, 
and $G$ is commutative if and only if $m_{ij}\leq 2$, for all $i\neq j$. Thus, if $m_{ij} > 2$, the involutions 
$r_i$ and $r_j$ do not commute, while if we have a group $G$ where 
all $m_{ij}\leq 2$, then the group is commutative. For example, the Klein four-group 
$$
K_4 = {\cal D}_2 = \langle r_1, r_2 ~|~ r_1^2 = r_2^2 = 1, ~ (r_1r_2)^2 = 1\rangle \cong C_2\times C_2
$$
is a Coxeter group ($C_2$ stands for the cyclic group of order two) and can be 
embedded into the Riordan group ${\cal R(\comp)}$ using the Riordan arrays 
$$
r_1:= P_0 = (1,\,-t),~~~ \mbox{and} ~~~ r_2 := (-1\,,t).
$$
Another example of a Coxeter group, which can be embedded into ${\cal R(\comp)}$ is the 
infinite dihedral group ${\cal D}_{\infty}$. This group is isomorphic to the semidirect 
product $\inte\rtimes C_2$ (see \S 2.2 of \cite{Robinson}), and has a presentation
$$
{\cal D}_{\infty} \cong \langle s, t ~|~ s^2 = t^2 = 1 \rangle.
$$
We prove a bit more general result. Let ${\mathbb D}\subset \comp$ be an integral 
domain, denote by ${\mathbb D}^*$ the multiplicative group of units of $\mathbb D$, and consider 
the subgroup $\Delta(2,{\mathbb D})$ of ${\rm GL}(2,{\mathbb D})$ consisting of the matrices
$$
\begin{pmatrix}
1 & 0\\
b & a\\
\end{pmatrix}, ~~ \mbox{where} ~~ b\in {\mathbb D}, \, a\in {\mathbb D}^*.
$$
Define a map from $\Delta(2,{\mathbb D})$ into the Lagrange subgroup $\L$ of the 
Riordan group $\R(\comp)$ by
\begin{equation}
\label{monom1}
f:  \begin{pmatrix}
1 & 0\\
b & a\\
\end{pmatrix} \longmapsto 
\left(1,\, \frac{ax}{1 + bx}\right)\in \L.
\end{equation}
It is easy to check that this $f$ is a monomorphism. Next lemma shows that $\Delta(2,{\mathbb D})$ 
is isomorphic to a semidirect product of ${\mathbb D}$  and ${\mathbb D}^*$.

\begin{lemma}
\label{lem:1}
$\Delta(2,{\mathbb D})$ is isomorphic to the semidirect product of the additive group 
${\mathbb D}$ and multiplicative group of units ${\mathbb D}^*$, that is 
$$
\Delta(2,{\mathbb D}) \cong {\mathbb D} \rtimes_{\varphi} {\mathbb D}^*,
$$
where $\varphi: {\mathbb D}^*\to Aut({\mathbb D})$ is a homomorphism 
defined by $\varphi(a)(b) = \varphi_a(b) = ab$.
\end{lemma}
\begin{proof}
The binary operation $\circ$ on ${\mathbb D} \rtimes_{\varphi} {\mathbb D}^*$ is defined by 
\begin{equation}
\label{BinOp}
(b, a)\circ (y, x) = (b + \varphi_a(y), ax) = (b + ay, ax).
\end{equation}
Since we clearly have 
$$
{\mathbb D} \cong  \left\{
\begin{pmatrix}
1 & 0\\
b & 1\\
\end{pmatrix} 
~|~ b\in {\mathbb D}\right\} \lhd \Delta(2,{\mathbb D}), ~~~ \mbox{and} ~~~ {\mathbb D}^* \cong  \left\{
\begin{pmatrix}
1 & 0\\
0 & a\\
\end{pmatrix} 
~|~ a\in {\mathbb D}^* \right\},
$$
it is natural to associate the matrix 
$$
\begin{pmatrix}
1 & 0\\
b & a\\
\end{pmatrix}
$$
with the pair $(b,a)\in {\mathbb D} \rtimes_{\varphi} {\mathbb D}^*$. Then the automorphism 
$\varphi_a: {\mathbb D}\to {\mathbb D}$ corresponding to $a\in {\mathbb D}^*$ will be written as
$$
\varphi_a: 
\begin{pmatrix}
1 & 0\\
x & 1\\
\end{pmatrix} \to
\begin{pmatrix}
1 & 0\\
ax & 1\\
\end{pmatrix},
$$
and we can write the binary operation (\ref{BinOp}) in terms of the matrices as
$$
\left(\begin{pmatrix}
1 & 0\\
b & 1\\
\end{pmatrix},
\begin{pmatrix}
1 & 0\\
0 & a\\
\end{pmatrix}\right ) \circ
\left(\begin{pmatrix}
1 & 0\\
y & 1\\
\end{pmatrix},
\begin{pmatrix}
1 & 0\\
0 & x\\
\end{pmatrix}\right ) = 
\left(\begin{pmatrix}
1 & 0\\
b + ay & 1\\
\end{pmatrix},
\begin{pmatrix}
1 & 0\\
0 & ax\\
\end{pmatrix}\right ).
$$
Since in $\Delta(2,{\mathbb D})$ we have 
$$ 
\begin{pmatrix}
1 & 0\\
b & a\\
\end{pmatrix}\cdot \begin{pmatrix}
1 & 0\\
y & x\\
\end{pmatrix} = 
\begin{pmatrix}
1 & 0\\
b + ay & ax\\
\end{pmatrix},
$$
the proof is finished.
\end{proof}

\begin{Corollary}
There is a faithful representation of the infinite dihedral group ${\cal D}_{\infty}$ by Riordan arrays 
with complex coefficients.
\end{Corollary}
\begin{proof}
Take ${\mathbb D} =\inte$ and apply the monomorphism $f$ defined in (\ref{monom1}) to 
the group $\Delta(2,\inte)\cong \mathbb {\cal D}_{\infty}$.
\end{proof}

Our next result shows in particular, that finite non-commutative Coxeter groups in general, 
can not be embedded into ${\cal R(\comp)}$.

\begin{theorem}
Let $G$ be a Coxeter group with at least two generators 
$r_1$ and $r_2$, such that their product $r_1r_2$ has a finite order $m\geq 3$, i.e. 
$(r_1r_2)^m = 1$. Then there exists no monomorphism
$$
\mu: G \hookrightarrow {\cal R(\comp)}.
$$
\end{theorem}
\begin{proof}
Suppose there exists such a monomorphism $\mu$. We let $\sigma_i:= \mu(r_i)$, so 
$\sigma_i\in {\cal R}$ are the elements of the Riordan group ${\cal R}$, which satisfy the relations
\begin{equation}
\label{relations1}
\sigma_i^2 = (1,t), ~~ \mbox{and} ~~ (\sigma_1\sigma_2)^{m} = (1,t),  ~m\geq 3.
\end{equation}
Thus, we assume that there are two Riordan involutions the product of which has order 
$m\geq 3$. Take the f.p.s. $g_i(t)$ and $f_i(t)$ such that
$$
\sigma_i = \bigl(g_i(t),\,f_i(t)\bigr), ~i\in\{1,2\}.
$$
Using the group operation in ${\cal R}$ together with $\sigma_i^2 = 1$, we see that  
$f_i$ and $g_i$ satisfy 
\begin{equation}
\label{twoID}
f_i(f_i(t)) = t, ~~ \mbox{and} ~~ g_i(t)g_i(f_i(t)) = 1.
\end{equation}
Let us write each $f_i(t)$ as a series expansion 
$$
f_1(t) = \sum\limits_{i\geq 1} a_it^i, a_1\neq 0 ~~~\mbox{and}  ~~~ 
f_2(t) = \sum\limits_{i\geq 1} b_it^i, b_1\neq 0.
$$
In order to satisfy the first identity of (\ref{twoID}), we must have $a_1^2 = b_1^2 = 1$, and it is a simple 
math induction exercise to show that if $a_1 = 1$ then $f_1(t) = t$. This is an element of compositional 
order one, and for $\sigma_1$ to be a proper (i.e. not the identity) involution, we must have $\sigma_1 = (-1,\,t)$. But 
$ (-1,\,t)$ commutes with every element of ${\cal R}$, and can't be $\mu(r_1)$, since $r_1$ 
and $r_2$ do not commute. Therefore we can assume that $a_1 = b_1 = -1$, and write 
$$
f_1(t) = -t + a_kt^k + a_{k+1}t^{k+1} + \cdots, ~ \mbox{and} ~
f_2(t) = -t + b_lt^l + b_{l+1}t^{l+1} + \cdots, 
$$
where $a_kb_l\neq 0$. Hence $\exists q > 1$, such that
$$
f_1(f_2(t)) = t + c_qt^q + c_{q+1}t^{q+1} + \cdots 
$$
with $c_q \neq 0$. Now, according to Theorem 2.3 of Jennings \cite{Jennings}, 
for the composition $f_1(f_2(t))$ and arbitrary $m\in\natu$, we would have the $m$th 
compositional power 
\begin{equation}
\label{HighOrder}
\bigl(f_1(f_2(t))\bigr)^{(m)} = t + mc_qt^q + (\mbox{higher powers of} ~ t)  \neq t.
\end{equation}
It implies that the identity $(f_1\circ f_2)^{(m)} = Id$ required for (\ref{relations1}) is impossible, 
and finishes the proof by contradiction. Note that \eqref{HighOrder} can be proved 
directly by using induction on $m$.
\end{proof}

Here are a few corollaries of this result (compare the first one with Corollary 2.4 of 
Jennings \cite{Jennings}).

\begin{Corollary}
A product of two non-commuting Riordan involutions over $\comp$ has infinite order.
\end{Corollary}

\begin{Corollary}
For any $n\geq 3$, neither ${\cal S}_n$ nor ${\cal D}_n$ can be 
embedded into ${\cal R(\comp)}$.
\end{Corollary}

\begin{Corollary}
The projective linear group ${\rm PGL}(2,\inte)$ can not be embedded into the Riordan group 
${\cal R(\comp)}$.
\end{Corollary}
\begin{proof}
It follows from Example 2 in \S 5.1 of Humphreys \cite{Humphreys}, where it is shown that ${\rm PGL}(2,\inte)$ is 
isomorphic to a Coxeter group
$$
{\rm PGL}(2,\inte) \cong \langle r_1,r_2,r_3~|~ r_i^2 = (r_1r_2)^3 = (r_1r_3)^2=1\rangle.
$$
\end{proof}

\noindent One may wonder here, if a product of more than two non-commuting Riordan 
involutions over $\comp$ can have a finite order. The generalized Pascal's matrices 
(\ref{GenPascal}) give the affirmative answer to this question. Since 
$\sum\limits_{k = 1}^{2n+1} (-1)^{k-1} k  = n+1$, the induction shows that the 
product of $2n+1$ Riordan involutions $P_b$ with $b\in\{1,2,\ldots, 2n+1\}$, 
$$
P_1\cdot P_2 \cdot\ldots \cdot P_{2n+1} = P_{n+1}
$$ 
is again the Riordan involution, so it has order $2$. Finally, the product of four non-commuting 
involutions can have order 1, since $P_1\cdot P_2\cdot P_5\cdot P_4 = (1,\,t)$.


\section{Embeddability into ${\mathcal R}(\inte_n)$}

Formula (\ref{HighOrder}) implies that if we had $mc_q = 0$, the identity 
$(f_1\circ f_2)^{(m)} = Id$ could be true. This observation suggests to consider 
the Riordan group with coefficients in the ring of integers modulo $n$, that is 
$\inte_n = \{0,1,\ldots,n-1\}$, and the so-called {\sl modular representations}. The elements of 
$\inte_n$ are the residue classes, but to simplify the notations we denote them by the 
corresponding integers. To have a group we need every element to be invertible, 
which means that we actually need to impose 
an extra condition on our elements $\bigl(g(t),\,f(t)\bigr)\in {\cal R}(\inte_n)$. We must 
require $g(0)$ and $f'(0)$ to be units of the ground ring. This condition will guarantee that the 
inverse $\bigl(g(t),\,f(t)\bigr)^{-1}\in {\cal R}(\inte_n)$ exists.
Our next result shows that for an arbitrary integer $n\geq 3$, the dihedral group of order $2n$ 
has many faithful representations in the Riordan group over $\inte_n$. 

\begin{theorem}
Fix two integers $n\geq 3$ and $s \geq 1$, and take arbitrary $a,b \in\{0,1,\ldots,n-1\}$, 
with $a^2 = 1$. Then the subgroup of ${\cal R}(\inte_n)$ generated 
by two Riordan matrices
\begin{equation}
\label{PascalZn}
r_1 = \left(\frac{a}{(1 + (b+1)t)^s},\,\frac{(n-1)t}{1+ (b+1)t}\right) ~ \mbox{and} ~ 
r_2 = \left(\frac{a}{(1 + bt)^s},\,\frac{(n-1)t}{1+ bt}\right)
\end{equation}
is isomorphic to the dihedral group ${\cal D}_n$ of order $2n$.
\end{theorem}
\begin{proof}
Using $nt = 0$ in $\inte_n$, a quick computation shows that each $r_i$ is an involution 
in ${\cal R}(\inte_n)$.  Similar computations give
$$
r_1r_2 = \left(\frac{1}{(1 + t)^s},\,\frac{t}{1 + t}\right) ~~ \mbox{and} ~~ 
r_2r_1 = \left(\frac{1}{(1 + (n-1)t)^s},\,\frac{t}{1 + (n-1)t}\right).
$$
In particular, we see that $r_1$ and $r_2$ do not commute, since $n\geq 3$. Moreover, the 
induction shows that for any $k\in\natu$,
$$
(r_1r_2)^k = \left(\frac{1}{(1 + kt)^s},\,\frac{t}{1 + kt}\right)
$$
and hence 
\begin{equation}
\label{R2R1R2}
r_2(r_1r_2)^k = \left(\frac{a}{(1 + (b-k)t)^s},\,\frac{(n-1)t}{1 + (b-k)t}\right),
\end{equation}
which implies that the product $r_1r_2$ has order $n$ in ${\cal R}(\inte_n)$. We also 
see from the last two equalities that $2n$ Riordan arrays 
$$
\{(r_1r_2)^k,~ r_2(r_1r_2)^k\},~k\in\{0,1,\ldots,n-1\}
$$
are all distinct. Thus, the subgroup generated by $r_1$ and 
$r_2$ is isomorphic to the group given by the presentation (\ref{SnDn}).
\end{proof}

\noindent The groups ${\cal D}_3$ and ${\cal S}_3$ are isomorphic, and we 
immediately obtain the following
\begin{Corollary}
The symmetric group ${\cal S}_3$ of degree 3 has a faithful representation 
by the Riordan arrays over the Galois field $\mathbb{GF}(3^q),\forall q\in\natu$.
\end{Corollary}

To be more specific, let us take $a = b = s = 1$ in (\ref{PascalZn}) and 
give a particular example of a faithful representation of ${\cal S}_3$ in terms of the 
Riordan arrays with coefficients in $\mathbb{GF}(3) = \inte_3$. Since ${\cal D}_3$ 
has three non-trivial involutions $\{r_2,~r_2r_1r_2,~r_2(r_1r_2)^2 = r_1\}$, to be in 
agreement with (\ref{S3Present}), we consider the following correspondence 
among the nonidentity elements: 
$$
r_2 \leftrightarrow \sigma_1,~ r_1 \leftrightarrow \sigma_2, ~ r_2r_1r_2 \leftrightarrow 
\sigma_1\sigma_2\sigma_1,~ r_1r_2 \leftrightarrow \sigma_2\sigma_1, ~ 
r_2r_1 \leftrightarrow \sigma_1\sigma_1.
$$
Thus, using (\ref{PascalZn}) we obtain the following faithful representation of ${\cal S}_3$ by 
Riordan matrices over $\inte_3$.
\begin{equation}
\label{S3RZ3}
\sigma_1 \leftrightarrow 
\begin{pmatrix}
1 & 0 & 0 & 0 & 0\cdots \\
2 & 2 & 0 & 0 & 0\cdots \\
1 & 2 & 1 & 0 & 0\cdots \\
2 & 0 & 0 & 2 & 0 \cdots \\
1 & 1 & 0 & 1 & 1 \cdots \\
\vdots & \vdots & \vdots & \vdots & \vdots \ddots
\end{pmatrix}, ~~ \sigma_2 \leftrightarrow 
\begin{pmatrix}
1 & 0 & 0 & 0 & 0\cdots \\
1 & 2 & 0 & 0 & 0\cdots \\
1 & 1 & 1 & 0 & 0\cdots \\
1 & 0 & 0 & 2 & 0 \cdots \\
1 & 2 & 0 & 2 & 1 \cdots \\
\vdots & \vdots & \vdots & \vdots & \vdots \ddots
\end{pmatrix}, 
\end{equation}
$$
\sigma_2 \sigma_1 \leftrightarrow 
\begin{pmatrix}
1 & 0 & 0 & 0 & 0\cdots \\
2 & 1 & 0 & 0 & 0\cdots \\
1 & 1 & 1 & 0 & 0\cdots \\
2 & 0 & 0 & 1 & 0 \cdots \\
1 & 2 & 0 & 2 & 1 \cdots \\
\vdots & \vdots & \vdots & \vdots & \vdots \ddots
\end{pmatrix}, ~~ 
\sigma_1 \sigma_2 \leftrightarrow 
\begin{pmatrix}
1 & 0 & 0 & 0 & 0\cdots \\
1 & 1 & 0 & 0 & 0\cdots \\
1 & 2 & 1 & 0 & 0\cdots \\
1 & 0 & 0 & 1 & 0 \cdots \\
1 & 1 & 0 & 1 & 1 \cdots \\
\vdots & \vdots & \vdots & \vdots & \vdots \ddots
\end{pmatrix},
$$
and 
$$
\sigma_2\sigma_1 \sigma_2 \leftrightarrow 
\begin{pmatrix}
1 & 0 & 0 & 0 & 0\cdots \\
0 & 2 & 0 & 0 & 0\cdots \\
0 & 0 & 1 & 0 & 0\cdots \\
0 & 0 & 0 & 2 & 0 \cdots \\
0 & 0 & 0 & 0 & 1 \cdots \\
\vdots & \vdots & \vdots & \vdots & \vdots \ddots
\end{pmatrix}, ~~ 
e\leftrightarrow 
\begin{pmatrix}
1 & 0 & 0 & 0 & 0\cdots \\
0 & 1 & 0 & 0 & 0\cdots \\
0 & 0 & 1 & 0 & 0\cdots \\
0 & 0 & 0 & 1 & 0 \cdots \\
0 & 0 & 0 & 0 & 1 \cdots \\
\vdots & \vdots & \vdots & \vdots & \vdots \ddots
\end{pmatrix}.
$$

Each Riordan array naturally reduces to a lower-triangular square $n\times n$ matrix 
by taking the first $n$ rows and columns. Moreover, it is clear  
that if we reduce this way every element of the Riordan group ${\cal R}(\mathbb K)$, 
we will get a group of lower-triangular $n\times n$ matrices. The common 
notation for such a reduced group is ${\cal R}_n(\mathbb K)$ and this {\sl reduction process} 
is an epimorphism, which we denote by $\pi_n:{\cal R}(\mathbb K)  
\twoheadrightarrow {\cal R}_n(\mathbb K)$. ${\cal R}_n(\mathbb K)$ is called the {\sl $n$-th 
truncation of ${\cal R}(\mathbb K)$}, and is obviously a subgroup of the corresponding 
general linear group $\mbox{GL}(n,\mathbb K)$. Therefore, 
if we have a linear representation $\rho$ of a group $G$ by Riordan arrays with 
coefficients in a field $\mathbb K$, we naturally obtain the {\sl induced representation} of 
degree $n$ in $\mbox{GL}(n,\mathbb K)$ for an arbitrary $n\geq 1$, via 
\begin{equation}
\label{IndRepr}
\rho_n: G \stackrel{\rho}\lra {\cal R}(\mathbb K) \stackrel{\pi_n}\twoheadrightarrow 
{\cal R}_n(\mathbb K)\hookrightarrow \mbox{GL}(n,\mathbb K).
\end{equation}

Now, if we take for example, the representation (\ref{S3RZ3}), and denote it by 
${}_{(1,1,1)}\rho$, where the subindex $(1,1,1)$ on the left of $\rho$ indicates that 
we used correspondingly $a = 1$, $b = 1$ and $s = 1$ in (\ref{PascalZn}), we can 
consider the induced representation 
$$
{}_{(1,1,1)}\rho_3 : {\cal S}_3 \lra \mbox{GL}(3,\inte_3),
$$
and compare it with representation by the permutation $3\times 3$ matrices 
(recall (\ref{S3Present})). By the comparison we mean the {\sl equivalence} 
of degree-$n$ representations in the following standard way: Let $V^n$ be a vector 
space of dimension $n$ over a field $\mathbb K$. Two representations 
$\varphi,~\psi: G\lra \mbox{GL}(V)$ are said to be {\sl equivalent} if there is an automorphism 
$T:V^n \to V^n$ such that $\varphi(g) = T\circ\psi(g)\circ T^{-1}$ for all $g\in G$. 

Representation (\ref{S3Present}) has ${\cal S}_3$-invariant subspace spanned by the 
vector $\langle 1,1,1\rangle$, and therefore can be reduced to a degree-2 representation 
(see \S 1.5 of \cite{Sagan}). Here is this degree-2 representation with coefficients in $\inte_3$:
\begin{equation}
\label{S3Presentby2}
e\leftrightarrow \begin{pmatrix}
1 & 0 \\
0 & 1\\
\end{pmatrix}, ~~ \sigma_1  \leftrightarrow 
\begin{pmatrix}
2 & 2 \\
0 & 1\\
\end{pmatrix},  ~~ \sigma_2 \leftrightarrow 
\begin{pmatrix}
1 & 0 \\
2 & 2 \\
\end{pmatrix},
\end{equation}

$$
\sigma_1\sigma_2 \leftrightarrow 
\begin{pmatrix}
0 & 1 \\
2 & 2\\
\end{pmatrix}, ~~ \sigma_2 \sigma_1 \leftrightarrow 
\begin{pmatrix}
2 & 2 \\
1 & 0\\
\end{pmatrix},  ~~ \sigma_1\sigma_2 \sigma_1 \leftrightarrow 
\begin{pmatrix}
0 & 1 \\
1 & 0 \\
\end{pmatrix}.
$$
Thus, we can simplify a little bit our comparison task, and instead of ${}_{(1,1,1)}\rho_3$ 
consider the induced representation of degree two, i.e. ${}_{(1,1,1)}\rho_2$. Table below lists all the 
elements of ${\cal S}_3$ written as products of disjoint cycles, the corresponding $2\times 2$ 
matrices reduced from (\ref{S3Present}) (we denote this representation by $B$ as in 
\S 1.5 of \cite{Sagan}), and the matrices from ${}_{(1,1,1)}\rho_2$.

\begin{center}
\begin{tabular}{|c|c|c|c|c|c|c|}
\hline
 $g\in {\cal S}_3$ & $\sigma_1$ & $\sigma_2$ & $\sigma_1\sigma_2$ & $\sigma_2\sigma_1$ & 
 $\sigma_1\sigma_2\sigma_1$ & $e$ \\
\hline
\mbox{Cycles} & $(1,2)(3)$ & $(1,3)(2)$ & $(1,3,2)$ & $(1,2,3)$ & $(1)(2,3)$ & $(1)(2)(3)$ \\
\hline
\hline
$B(g)$ & $\begin{pmatrix}
2 & 2 \\
0 & 1\\ \end{pmatrix}$ & $\begin{pmatrix}
1 & 0 \\
2 & 2 \\ \end{pmatrix}$ & $\begin{pmatrix} 
0 & 1 \\
2 & 2\\ \end{pmatrix}$ & $\begin{pmatrix} 
2 & 2 \\
1 & 0\\ \end{pmatrix}$ & $\begin{pmatrix} 
0 & 1 \\
1 & 0 \\ \end{pmatrix}$ & $\begin{pmatrix} 
1 & 0 \\
0 & 1 \\ \end{pmatrix}$ \\
\hline
${}_{(1,1,1)}\rho_2(g)$ & $\begin{pmatrix}
1 & 0 \\
2 & 2 \\ \end{pmatrix}$ & $\begin{pmatrix}
1 & 0 \\
1 & 2\\ \end{pmatrix}$ & $\begin{pmatrix} 
1 & 0 \\
1 & 1\\ \end{pmatrix}$ & $\begin{pmatrix} 
1 & 0 \\
2 & 1\\ \end{pmatrix}$ & $\begin{pmatrix} 
1 & 0 \\
0 & 2 \\ \end{pmatrix}$ & $\begin{pmatrix} 
1 & 0 \\
0 & 1 \\ \end{pmatrix}$ \\
\hline
\end{tabular}
\end{center}
For the representations $B$ and ${}_{(1,1,1)}\rho_2$ to be equivalent, there should exist an 
invertible matrix 
$$
T = \begin{pmatrix}
x & y\\
z & w\\
\end{pmatrix} \in \mbox{GL}(2, \inte_3) 
$$
such that $T\cdot {}_{(1,1,1)}\rho_2(g) = B(g) \cdot T,~ \forall g\in {\cal S}_3$. 
Comparing the matrix products for the generators $\sigma_1$ and $\sigma_2$,
$$
\begin{pmatrix}
x & y\\
z & w\\
\end{pmatrix} \cdot \begin{pmatrix}
1 & 0\\
2 & 2\\
\end{pmatrix} = 
\begin{pmatrix}
2 & 2\\
0 & 1\\
\end{pmatrix} \cdot \begin{pmatrix}
x & y\\
z & w\\
\end{pmatrix}
$$
and
$$
\begin{pmatrix}
x & y\\
z & w\\
\end{pmatrix} \cdot \begin{pmatrix}
1 & 0\\
1 & 2\\
\end{pmatrix} = 
\begin{pmatrix}
1 & 0\\
2 & 2\\
\end{pmatrix} \cdot \begin{pmatrix}
x & y\\
z & w\\
\end{pmatrix},
$$
one sees that we must have $y = w = 0$. Therefore the invertible $T$ does not exist and 
hence these two representations are not equivalent. 

The next question, of course, is to see if any of the degree-2 induced representations 
of the form (\ref{PascalZn}) will be equivalent to the representation $B$. We go one step further, 
we classify all degree-2 representations ${}_{(a,b,s)}\rho_2$ by showing that the scalar $a$ 
completely determines the representation.

\begin{theorem}
The set of all degree-2 induced faithful representations 
${}_{(a,b,s)}\rho_2: {\cal S}_3\hookrightarrow {\cal R}_2(\inte_3)$ consists of exactly two 
nonequivalent representations. That is, ${}_{(a,b,s)}\rho_2 = {}_{(a,b',s')}\rho_2$ 
for all $b,b'\in\inte_3$ and $s,s'\in\{1,2\}$, and 
$$
{}_{(1,b,s)}\rho_2 \nsim {}_{(2,b,s)}\rho_2.
$$
\end{theorem}
\begin{proof}
To show that there are only two different representations 
${\cal S}_3\hookrightarrow {\cal R}_2(\inte_3)$ we can use a counting 
argument. Indeed, following the approach we used to prove 
Lemma \ref{lem:1}, one can show that the truncated Riordan group 
$$
{\cal R}_2(\inte_3) = \left\{ \left.\begin{pmatrix}
a & 0\\
b & c \end{pmatrix} ~ \right| ~  a,b,c\in\inte_3,~ \mbox{and} ~ a,c\neq 0\right\} 
$$
is isomorphic to ${\cal D}_6$. Therefore the image of any embedding 
$\rho_2 : {\cal S}_3\hookrightarrow {\cal R}_2(\inte_3)$ will be a normal subgroup of index 2.
Here is the list of all elements of ${\cal R}_2(\inte_3)$:
\begin{itemize}
\item Order 6: 
$$
z_1 = \begin{pmatrix}
2 & 0\\
1 & 2\\ \end{pmatrix}, ~~ \mbox{and} ~~ z_2 = \begin{pmatrix}
2 & 0\\
2 & 2\\ \end{pmatrix}
$$

\item Order 3:
$$
y_1 = \begin{pmatrix}
1 & 0\\
1 & 1\\ \end{pmatrix}, ~~ \mbox{and} ~~ y_2 = \begin{pmatrix}
1 & 0\\
2 & 1\\ \end{pmatrix}
$$

\item Involutions outside the center $Z\bigl({\cal R}_2(\inte_3)\bigr)$:
$$
I_1 = \begin{pmatrix}
1 & 0\\
0 & 2\\ \end{pmatrix}, ~~ I_2 = \begin{pmatrix}
1 & 0\\
1 & 2\\ \end{pmatrix}, ~~ I_3 = \begin{pmatrix}
1 & 0\\
2 & 2\\ \end{pmatrix}, ~~ \mbox{and} ~~ 2I_1, ~ 2I_2, ~ 2I_3
$$

\item The center $Z\bigl({\cal R}_2(\inte_3)\bigr)$: 
$$
E = \begin{pmatrix}
1 & 0\\
0 & 1\\ \end{pmatrix}, ~~ \mbox{and} ~~ 2E = \begin{pmatrix}
2 & 0\\
0 & 2\\ \end{pmatrix}.
$$
\end{itemize}
Now, since 
$$
\frac{a}{(1 + bt)^s} = a(1 - bst) + (\mbox{higher powers of} ~ t) = a + 2abst + \ldots,
$$
and 
$$
\frac{2t}{(1 + bt)} = 2t + (\mbox{higher powers of} ~ t),
$$
we have two generating involutions of the image 
$ {}_{(a,b,s)}\rho_2({\cal S}_3) \subset {\cal R}_2(\inte_3)$
\begin{equation}
\label{GenerA}
r_1 = a\begin{pmatrix}
1 & 0\\
2(1+b)s & 2\\
\end{pmatrix}, ~~~ r_2 = a\begin{pmatrix}
1 & 0\\
2bs & 2\\
\end{pmatrix}
\end{equation}
Since $r_1\neq r_2$ we can not have $s = 0$, and hence $a$ and $ s \in \{1,2\} = \inte_3^*$.
Thus, if $a = 1$ we have three options to choose our $r_1$ and $r_2$ from, and each 
of the choices will give us the same embedding 
$$
{}_{(1,b,s)}\rho_2({\cal S}_3) = \{I_1, I_2, I_3, y_1, y_2, E\}\subset {\cal R}_2(\inte_3).
$$
This proves that ${}_{(1,b,s)}\rho_2 = {}_{(1,b',s')}\rho_2$ for all choices of $b,b'\in\inte_3$ and 
$s,s' \in\inte^*_3$. Similarly, ${}_{(2,b,s)}\rho_2 = {}_{(2,b',s')}\rho_2$ for all choices of 
$b,b'\in\inte_3$ and $s,s' \in\inte^*_3$, and therefore all these six choices 
will give us the same embedding 
$$
{}_{(2,b,s)}\rho_2({\cal S}_3) = \{2I_1, 2I_2, 2I_3, y_1, y_2, E\}\subset {\cal R}_2(\inte_3).
$$
To show that 
$$
{}_{(1,b,s)}\rho_2 \nsim {}_{(2,b,s)}\rho_2,
$$
it is enough to consider two particular triples giving different embeddings. We take 
${}_{(1,0,1)}\rho_2$ and ${}_{(2,0,1)}\rho_2$. These embeddings are generated by 
$$
\left\langle r_1= \begin{pmatrix}
1 & 0\\
1 & 2\\ \end{pmatrix}, r_2 = \begin{pmatrix}
1 & 0\\
0 & 2\\ \end{pmatrix}\right\rangle ~~ \mbox{and} ~~ 
\left\langle r'_1= \begin{pmatrix}
2 & 0\\
2 & 1\\ \end{pmatrix}, r'_2 = \begin{pmatrix}
2 & 0\\
0 & 1\\ \end{pmatrix} \right\rangle
$$
respectively.
Let us denote a possible nonsingular transformation matrix $T$ by 
$
\begin{pmatrix}
x & y\\
z & w\\
\end{pmatrix}
$ as above. It is clear that we only need to show the impossibility of
$$
T\cdot r_i = r'_i \cdot T
$$
for each $i\in\{1,2\}$. If we use $i =1$ we'd have the equalities
$$
\begin{pmatrix}
x & y\\
z & w\\
\end{pmatrix}\cdot \begin{pmatrix}
1 & 0\\
1 & 2\\
\end{pmatrix} = \begin{pmatrix}
x+y & 2y\\
z+w & 2w\\
\end{pmatrix} = 
\begin{pmatrix}
2 & 0\\
2 & 1\\
\end{pmatrix}\cdot \begin{pmatrix}
x & y\\
z & w\\
\end{pmatrix} = \begin{pmatrix}
2x & 2y\\
2x + z & 2y + w\\
\end{pmatrix},
$$
which imply that $y = x$ and $w = 2x$. Using now $i = 2$ and the matrix $T$ where $y = x$ and 
$w = 2x$, we obtain
$$
\begin{pmatrix}
x & x\\
z & 2x\\
\end{pmatrix}\cdot \begin{pmatrix}
1 & 0\\
0 & 2\\
\end{pmatrix} = \begin{pmatrix}
x & 2x\\
z & x\\
\end{pmatrix} = 
\begin{pmatrix}
2 & 0\\
0 & 1\\
\end{pmatrix}\cdot \begin{pmatrix}
x & x\\
z & 2x\\
\end{pmatrix} = \begin{pmatrix}
2x & 2x\\
z & 2x\\
\end{pmatrix}.
$$
Therefore we must have $2x = x \Leftrightarrow x = 0$, which makes $T$ singular.
Thus the equivalence class is completely determined by the constant $a$, and 
the proof is finished.
\end{proof}

\begin{Corollary} Degree-2 representation of ${\cal}S_3$ by the matrices (\ref{S3Presentby2}) 
reduced from the permutation matrices is equivalent to the representation ${}_{(2,0,1)}\rho_2$. 
\end{Corollary}
\begin{proof}
Use the transformation matrix 
$$
T = \begin{pmatrix} 0 & 1\\ 1 & 1\end{pmatrix}, 
$$
which will satisfy $T\cdot r'_i = B(\sigma_i)\cdot T$ for each $i\in\{1,2\}$ 
(recall the table above for $B(\sigma_i)$).
\end{proof}

At this point it seems natural to call a representation of a group $G$ by truncated 
Riordan $n\times n$ matrices obtained from the composition of two homomorphisms 
$$
\varphi_n: G \stackrel{\rho}\lra {\cal R}(\mathbb K) \stackrel{\pi_n}\twoheadrightarrow 
{\cal R}_n(\mathbb K)
$$
as the {\sl degree-$n$ Riordan representation $\varphi_n$ of $G$ over $\mathbb K$}. 
Using this new terminology, we can say that in Theorem 10, we classify the {\sl degree-2 
Riordan representations of ${\cal S}_3$ over $\inte_3$}. 

~

\noindent We would like to close our discussion here with a few related questions, 
which we have no doubt are also on the reader's mind.

\begin{itemize}
\item Question 1: What is the classification of the degree-$n$ Riordan representations of 
${\cal S}_3$ over $\inte_3$ for $n\geq 3$?

\item Question 2: What is the classification of the degree-$n$ Riordan representations of 
${\cal S}_3$ over the Galois field $\mathbb{GF}(3^q)$ for $q,n\geq 2$? 

\item Question 3: Are there any faithful Riordan representations of ${\cal S}_n$ 
over a field of finite characteristic when $n\geq 4$?

\item Question 4:  What is the set of all nonequivalent degree-$k$ induced faithful 
representations $\varphi_k: S_{n}\hookrightarrow  R_{k}({\mathbb Z}_{n})$ for 
$n \geq 4$ and $k \leq n$, if such exist?

\item Question 5: What is the classification of the degree-$n$ Riordan representations of 
${\cal D}_n$ over the ring $\inte_n$ when $n\geq 4$?

\end{itemize}

\section*{Acknowledgement}

The authors sincerely thank the organizers of the Special Session on Riordan Arrays of the 
AMS 2024 Spring Eastern Sectional Meeting for providing an effective platform for academic 
exchange and cooperation. The authors particularly thank Lou Shapiro and  Alex Burstein 
for the fruitful discussions about involutions and pseudo-involutions during the conference, 
and also the referees for suggesting a shortened proof of Theorem 3.

\end{document}